\documentclass[a4paper,11pt,makeidx]{amsart}
\usepackage{amscd}
\usepackage{xypic}  
\usepackage{amssymb}
\usepackage{amsthm}
\usepackage{epsf}
\makeindex
\newtheorem{thm}{Theorem}[section]

\newtheorem{lemma}[thm]{Lemma}

\newtheorem{prop}[thm]{Proposition}

\newtheorem{defn}[thm]{Definition}

\newtheorem{rem}[thm]{Remark}

\renewcommand{\proofname}{Proof}

\def\Im{\operatorname{Im}}

\def\max{\operatorname{max}}
\def\length{\operatorname{length}}
\def\c1{\operatorname{c_1}}
\def\c2{\operatorname{c_2}}
\def\Cliff{\operatorname{Cliff}}

\def\gon{\operatorname{gon}}
\def\gengon{\operatorname{gengon}}
\def\mingon{\operatorname{mingon}}

\def\PP{{\mathbb P}}

\def\O{{\mathcal O}}
\def\I{{\mathcal J}}

\def\E{{\mathcal E}}

\def\J{{\mathfrak J}}
\def\*{\otimes}

\def\eqv{\equiv}
\def\sub{\subseteq}

\def\+{\oplus}                   
\def\*{\otimes}                  
\def\hpil{\longrightarrow}       

\def\Pic{\operatorname{Pic}}

\def\Supp{\operatorname{Supp}}

\def\Bs{\operatorname{Bs}}

\begin{document}

\title[Projective normality and the ideal of an Enriques surface.]{Projective normality and the generation of the ideal of an Enriques surface}

\author[A.L. Knutsen and A.F. Lopez]{Andreas Leopold Knutsen and Angelo Felice Lopez*}

\address{\hskip -.43cm Andreas Leopold Knutsen, Department of Mathematics, University of Bergen, Postboks 7800,
5020 Bergen, Norway. e-mail {\tt andreas.knutsen@math.uib.no}}

\address{\hskip -.43cm Angelo Felice Lopez, Dipartimento di Matematica e Fisica, Universit\`a di Roma 
Tre, Largo San Leonardo Murialdo 1, 00146, Roma, Italy. e-mail {\tt lopez@mat.uniroma3.it}}

\thanks{* Research partially supported by the MIUR national project ``Geometria delle variet\`a algebriche" PRIN 2010-2011.}

\thanks{{\it 2000 Mathematics Subject Classification} : Primary 14J28. Secondary 14H51, 14C20.}

\begin{abstract}
We give necessary and sufficient criteria for a smooth Enriques surface $S \subset \PP^r$ to be scheme-theoretically an intersection of quadrics.
Moreover we prove in many cases that, when $S$ contains plane cubic curves, the intersection of the quadrics containing $S$ is the union of $S$ and the
$2$-planes spanned by the plane cubic curves. We also give a new (very quick) proof of the projective normality of $S$ if $\deg S \geq  12$.
\end{abstract}

\maketitle

\section{Introduction}
\label{intro}

Even though it is a very basic question, it is in general difficult, given a projective variety $X \subset \PP^r$, to be able to tell about its equations, as they depend in a nontrivial way on the geometry and often on the moduli of $X$. An emblematic example of this is the case of curves, when the equations 
are strictly related to its Clifford index and the investigation of this problem has led to very important results and conjectures, such as Green's conjecture \cite{gr} and Voisin's theorem \cite{v1, v2}.

In the present paper we deal with the problem of finding the degrees of the equations of Enriques surfaces $S \subset \PP^r$. As a matter of fact it is not difficult to see that one can give an answer as soon as one knows it for a general hyperplane section $C_P = S \cap H$ passing through a fixed point $P \in \PP^r$. Now for curves $C$ important results, proved by Green and Lazarsfeld \cite[Thm. 1]{gl},  \cite[Prop. 2.4.2]{la} (we use also \cite[Thm. 1.3]{ls}) come to help: if $\deg C \geq 2g(C) +1- 2h^1(\O_C(1)) - \Cliff C$ then $C$ is projectively normal and if $\deg C$ is at least one more, then $C$ is scheme-theoretically cut out by quadrics unless it has a trisecant line. 
It is therefore of crucial importance to know the Clifford index of curves on Enriques surfaces. Now in \cite[Thm. 1.1]{kl3} we proved that $\Cliff C = \gon(C) - 2$, where $\gon(C)$ is the gonality of $C$. A careful study of $\gon(C_P)$ leads us then to the main result of this paper (for the function $\phi(L)$ see Def. \ref{phi}):

\begin{thm}   
\label{ideale}
Let $S \subset \PP^r$ be a smooth linearly normal nondegenerate Enriques surface and let $L = \O_S(1)$. Then
\begin{itemize}
\item[(i)] $S$ is projectively normal if $L^2 \geq 12$.
\item[(ii)] $S$ is scheme-theoretically cut out by quadrics if and only if $\phi(L) \geq 4$;
\item[(iii)] If $L^2 \geq 18$ and $\phi(L) = 3$ and $L$ is not of special type (see Def. \ref{18sp}), then the intersection of the quadrics containing $S$ is the union of $S$ and the
$2$-planes spanned by cubic curves $E \subset S$ such that $E^2=0$, $E.L=3$.
\end{itemize}
\end{thm}

We recall that for $L$ to be very ample we need $L^2 \geq 10$ and $\phi(L) \geq 3$. Now (ii) and (iii) improve \cite[Thm. 1.3]{glm2}, while (i) gives a very quick new proof of \cite[Thm. 1.1]{glm1} (except for the case $L^2 = 10$). Note that the proof of (i) consists only of the first five lines of Section \ref{sec:genid}, which do not depend on Section \ref{sec:preliminari}. Also in case (iii) we have some partial results when $L^2 = 16$ (see Remark \ref{rm}).

We give some preliminary results in Section \ref{sec:preliminari}, among wich a generalization of \cite[Cor. 1.6]{kl1}. Then Section \ref{sec:genid} is devoted to the proof of Theorem \ref{ideale}. 

\section{Preliminary results}
\label{sec:preliminari}

\begin{defn}
We denote by $\sim$ (resp. $\eqv$) the linear (resp. numerical) equivalence of divisors or line bundles on a smooth surface. If $V \subseteq H^0(L)$ is a linear system, we denote its {\bf base scheme} by $\Bs |V|$. A {\bf nodal} curve on an Enriques surface $S$ is a smooth rational curve contained in $S$. A {\bf nodal cycle} is a divisor $R>0$ such that, for any $0 < R' \leq R$ we have $(R')^2 \leq -2$.
\end{defn} 
Now recall from \cite{cd} the following
\begin{defn}
\label{phi}
Let $L$ be a line bundle on an Enriques surface $S$ such that $L^2 > 0$. Set
\[ \phi(L) = \inf \{|F.L| \; : \; F \in \Pic S, F^2 = 0, F \not\eqv 0\}. \] 
\end{defn} 

In \cite{kl1} we proved the following result about the variation of the gonality in linear systems on Enriques surfaces

\begin{prop}  \cite[Cor. 1.6]{kl1}
\label{vargon}
Let $|L|$ be a base-component free complete linear system on an Enriques surface such that $L^2
> 0$. Let $\gengon |L|$ denote the gonality of a general smooth curve in $|L|$ and $\mingon |L|$ 
denote the minimal gonality of a smooth curve in $|L|$. Then
\[\gengon |L|-2 \leq \mingon |L| \leq \gengon |L|. \] 

\noindent Moreover if equality holds on the left, then $\phi(L) \geq \lceil \sqrt{\frac{L^2}{2}}
\rceil$.
\end{prop}

We will need the ensuing generalization 

\begin{prop} 
\label{more}
Let $|L|$ be a base-component free complete linear system on an Enriques surface such that $L^2 > 0$. Let $k:= \gengon |L|$ and assume that there is a smooth curve $C_0 \in |L|$ with $k_0 := \gon(C_0) \geq 2$.

If $k_0 = k - 2$ then there exists a line bundle $D_0 > 0$ such that either 
\begin{itemize}
\item[(i)] $D_0^2 = 2$, $D_0.L = k = 2 \phi(L)$ and $L^2 \geq 8\phi(L)-8$; or
\item[(ii)] $D_0^2 = 4$, $\phi(D_0) = 2$, $k = \mu(L) = D_0.L - 2$ and $L^2 \geq \max\{4k-8, 2k+4\}$. 
\end{itemize}

\noindent If $L^2 \geq 4k - 4$ and $k_0 = k - 1$ then there exists a line bundle $D_0 > 0$ such that either 
\begin{itemize}
\item[(iii)] $D_0^2 = 2$, $k \leq 2 \phi(L) \leq D_0.L \leq k + 1$ and  $L^2 \geq 4D_0.L-8$; or
\item[(iv)] $D_0^2 = 4$, $\phi(D_0) = 2$, $k + 2 \leq \mu(L) + 2 \leq D_0.L \leq k + 3$ and $L^2 \geq \max\{4D_0.L-16, 2D_0.L\}$.
\end{itemize}
\end{prop}

\begin{proof} Let $k_0 = k - \varepsilon$ with $\varepsilon = 1$ or $2$, so that $k \geq 3$ and then $\phi(L) \geq 2$. In particular $L$ is base-point free.
When  $\varepsilon = 2$ we have $k_0 \leq \lfloor \frac{L^2}{4} \rfloor$ by \cite[Lemma 5.1]{kl1}, so that $L^2 \geq
4k_0 \geq 8$. This also holds by hypothesis when  $\varepsilon = 1$.
By \cite[Rmk 4.9]{kl1} with $b = k_0 - 1$ we get $L^2 \geq 2k_0 - 2 + 2 \phi(L)$. Fix an ample divisor $H$ and let $A_0$ be a $g^1_{k_0}$ on $C_0$. By \cite[Prop. 3.1]{kl1}(a) we obtain an effective nontrivial decomposition $L \sim N + N'$ with $N.N' \leq k_0$ and $h^0(N') \geq 2$. Moreover we have $(N -
N')^2 = L^2 - 4N.N' \geq 0$ and $H.(N - N') \geq 0$, so that by  Riemann-Roch we have that either $N \eqv N'$ and $L \eqv 2N'$ or $L - 2N' = N
- N' \geq 0$. Consider the set
\begin{eqnarray*}
X_L & = & \{D \in \Pic(S) \ : \ h^0(D) \geq 2, D.(L - D) \leq k_0, \\ & & \ (L - 2D)^2
\geq 0 \ \mbox{and either} \ L - 2D \geq 0 \ \mbox{or} \ L \eqv 2D \}.
\end{eqnarray*}
Now $X_L \neq \emptyset$ since $N' \in X_L$, whence we can choose an element $D_0 \in X_L$ for which
$H.D_0$ is minimal. Note that $h^0(D_0) \geq 2$ implies that $D_0.L \geq 2\phi(L)$.  Set $G:= L - 2D_0$ so that we know that either $G \geq 0$ or $G \eqv 0$, whence $L.(L-2D_0) = L.G \geq 0$. Together with $(L - 2D_0)^2 \geq 0$ we get 
\begin{equation}
\label{zero}
L^2 \geq \max\{4D_0.L-4D_0^2, 2D_0.L\}.
\end{equation}
\noindent We first prove that we cannot have $D_0^2 \geq 6$. In fact if this is the case, pick 
any $F > 0$ such that $F^2 = 0$ and $F.D_0 = \phi(D_0)$. Then $(D_0 - F)^2 \geq 2$ and $D_0 - F > 0$ 
by \cite[Lemma 2.4]{kl1}, so that $h^2(D_0 - F) = 0$ whence $h^0(D_0 - F) \geq 2$ by Riemann-Roch.
Moreover
\[ (D_0 - F)(L - D_0 + F) = D_0.(L - D_0) - F.(L - 2D_0) \leq k_0 \]
since, by \cite[Lemma 2.1]{klvan}, we have that $F.(L - 2D_0) \geq 0$. For the same reason
\[ (L - 2(D_0 - F))^2 = (L - 2D_0)^2 + 4 F.(L - 2D_0) \geq 0. \]
Finally $L - 2(D_0 - F) = L - 2D_0 + 2F > 0$ both when $L - 2D_0 \geq 0$ and when $L \eqv 2D_0$. Now 
this contradicts the minimality of $D_0$. Therefore we have $D_0^2 \leq 4$.

\noindent We have
\begin{equation}
\label{uno}
k_0 = k - \varepsilon \leq 2\phi(L) - \varepsilon \leq D_0.L - \varepsilon = D_0^2 + D_0.(L - D_0) - \varepsilon \leq D_0^2 + k_0 - \varepsilon
\end{equation}
whence $D_0^2 \geq \varepsilon \geq 1$, so that $D_0^2 \geq 2$. 

\noindent Suppose $D_0^2 = 2$. If $\varepsilon = 2$ we have equality in \eqref{uno}, so that $D_0.L = k = 2\phi(L)$ and \eqref{zero} gives $L^2 \geq 8\phi(L)-8$, that is (i). If $\varepsilon = 1$ then \eqref{uno} gives  $k \leq 2 \phi(L) \leq D_0.L \leq k + 1$ and \eqref{zero} gives $L^2 \geq 4D_0.L-8$, that is (iii).

\noindent Now suppose $D_0^2 = 4$. We first exclude the case $\phi(D_0) = 1$. In fact, in the latter case, we can write $D_0
\sim 2F_1 +  F_2$ with $F_i > 0$, $F_i^2 = 0$, $i = 1, 2$, $F_1.F_2 = 1$ and we get
\begin{equation}
\label{due}
3 \phi(L) \leq D_0.L = D_0^2 + D_0.(L - D_0) \leq 4 + k_0 = 4 + k - \varepsilon \leq 4 + 2\phi(L) - \varepsilon
\end{equation}
whence $2 \leq \phi(L) \leq 4 - \varepsilon$. If $\varepsilon = 2$ then  $\phi(L) = 2$ and we have equality in \eqref{due}, so that $D_0.L=6$ and the Hodge index theorem applied to $D_0$ and $L$ then gives $L^2 \leq 8$, contradicting \eqref{zero}. If $\varepsilon = 1$ then  $\phi(L) \leq 3$ and if equality holds then we have equality in \eqref{due}, so that $D_0.L=9$ and $F_1.L = F_2.L = 3$. Now the Hodge index theorem applied to $F_1 + F_2$ and $L$ gives $L^2 \leq 18$, contradicting \eqref{zero}. Therefore $\phi(L)= 2$ and \eqref{due} gives $6 \leq D_0.L \leq 7$ and therefore $F_1.L = 2$, $2 \leq F_2.L \leq 3$. Now the Hodge index theorem applied to $F_1 + F_2$ and $L$ gives $L^2 \leq 8$ if $F_2.L = 2$ and $L^2 \leq 12$ if $F_2.L = 3$, both contradicting \eqref{zero}. 

\noindent Therefore we have proved that, when $D_0^2 = 4$, we have $\phi(D_0) = 2$. Since certainly $D_0 \not\eqv L$, we have, by  \cite[Lemma 5.1]{kl1},
\[ k \leq \mu(L) \leq D_0.L - 2 = 2 + D_0.(L - D_0) \leq 2 + k_0 = 2 + k - \varepsilon \]
so that, using \eqref{zero}, we deduce that either $\varepsilon = 2$ and $k = \mu(L) = D_0.L - 2$, $L^2 \geq \max\{4k-8, 2k+4\}$, that is (ii), or $\varepsilon = 1$ and $k + 2 \leq \mu(L) + 2 \leq D_0.L \leq k + 3$, $L^2 \geq \max\{4D_0.L-16, 2D_0.L\}$, that is (iv).
\end{proof}

For the sequel we will also need the following simple 
\begin{lemma} 
\label{18}
Let $|L|$ be a base-component free complete linear system on an Enriques surface $S$ such that $L^2 =18$ and $\phi(L)=3$. Then there exist $E>0, E_i > 0$, $1 \leq i \leq 3$, such that $E^2 = E_i^2 = 0$,  $E.E_i = E_i.E_j = 1$ for $i \neq j$ and either

\begin{itemize}
\item[(i)] $L \sim 3(E + E_1)$; or
\item[(ii)] $L \sim 2 E + E_1 + E_2 + E_3$. 
\end{itemize}
Moreover, in case (ii) we have that $H^1(E+E_1+K_S) = H^1(E+E_i+E_j) = 0$.
\end{lemma} 
\begin{proof} 
Let $E>0$ be such that $E^2 = 0$ and $E.L = 3$. By \cite[Lemma 2.4]{kl1} we can write $L \sim 3E + F$ with $F>0, F^2 = 0$ and $E.F = 3$. If $\phi(E+F) = 1$ then $F \eqv 3E_1$ for some $E_1>0$ such that $E_1^2 = 0$ and we are in case (i). If $\phi(E+F) = 2$ then we can write $E+F = E_1 + E_2 + E_3$ for some $E_i > 0$ with $E_i^2 = 0$ and $E.E_i =E_i.E_j = 1$ for $i \neq j$ and we are in case (ii). To prove that  $H^1(E+E_1+K_S) = 0$ it is enough, by \cite[Cor. 2.5]{klvan}, to see that $(E+E_1).\Delta \geq -1$ for every $\Delta > 0$ such that $\Delta^2=-2$. Now if $(E+E_1).\Delta \leq -2$ then $k:= -E_1.\Delta \leq -2$, whence, by \cite[Lemma 2.3]{kl1} we can write $E_1 = A + k \Delta$ with $A > 0, A^2=0$. As $L.\Delta > 0$ we can find a divisor $D \in\{E, E_2, E_3\}$ such that $D.\Delta \geq 1$, giving the contradiction $1 = D.E_1 = D.A + kD.\Delta \geq 2$. Similarly we can prove that $H^1(E+E_i+E_j) = 0$.
\end{proof} 
\begin{defn} 
\label{18sp}
Let $|L|$ be a base-component free complete linear system on an Enriques surface $S$ such that $L^2 =18$ and $\phi(L)=3$. We say that $L$ is {\bf of special type} if $L \sim 2 E + E_1 + E_2 + E_3$ with either $E_2+E_3-E_1 > 0$ or $E_2+E_3+K_S-E_1 > 0$.
\end{defn} 

\begin{rem} 
\label{unnodal}
Note that line bundles $L$ of special type exist only on nodal Enriques surfaces (that is that contain a smooth rational curve). In particular they do not exist on the general Enriques surface.
\end{rem} 

\section{Proof of Theorem \ref{ideale}.}
\label{sec:genid}

\begin{proof}
Observe first that, being $L$ very ample, we have $\phi(L) \geq 3$. Let $C_{\eta}$ be a general hyperplane section of $S$ and set $k = \gon (C_{\eta})$. To see (i) it is enough to prove that $C_{\eta}$ is projectively normal and, as $H^1(L_{|C_{\eta}}) = 0$ and $\deg L_{|C_{\eta}} = 2 g(C_{\eta}) - 2$, by \cite[Thm. 1]{gl}, we just need to prove that $\Cliff(C_{\eta}) \geq 3$. Now $\gon(C_{\eta}) \geq 5$ by \cite[Cor. 1.5]{kl1} whence $\Cliff(C_{\eta}) \geq 3$ by \cite[Thm. 1.1]{kl3}. This concludes the proof of (i).

To see (ii) note that if $\phi(L) = 3$, then $S$ contains plane cubics, therefore it has trisecant lines and it cannot be scheme-theoretically cut out by quadrics. This proves the ``only if" part in (ii). 

Vice versa suppose now $L^2 \geq 18$. We will prove (ii) and (iii) together. Let $P \in \PP^r$ be a point, let $|V_P|$ be the linear system cut out on $S$ by hyperplanes passing through $P$ and let $C_P$ be a general element of $|V_P|$ . To prove (ii) we will suppose that $\phi(L) \geq 4$ and $P \in S$, while to prove (iii) we will suppose that $\phi(L) = 3$, $P\not\in S$ and $P$ does not belong to the union of the 2-planes spanned by cubic curves $E \subset S$ such that $E^2=0$, $E.L=3$. 

Observe that, by \cite[Prop. 2.1]{glm2}, if $\phi(L) \geq 4$ then $C_P$ has no trisecant lines, while if $\phi(L) = 3$ then any trisecant line to $C_P$ belongs to some 2-plane spanned by a cubic curve $E \subset S$ such that $E^2=0$, $E.L=3$. Therefore to see (ii) one easily sees that  it is enough to prove that, if $P \in S$, then $C_P$ is scheme-theoretically cut out by quadrics at $P$ (see, for example, the proof of \cite[Thm. 1.3]{glm2}). Also, as $S$ is linearly normal, any quadric containing $C_P$ lifts to a quadric containing $S$, whence, to see (iii), it is enough to prove that the intersection of the quadrics containing $C_P$ is $C_P$ union all its trisecant lines. 

Now, to see (ii), by \cite[Prop. 2.4.2]{la} we see that $C_P$ is scheme-theoretically cut out by quadrics as soon as $\Cliff(C_P) \geq 4$, while, to see (iii), by \cite[Thm. 1.3]{ls}, we see that the intersection of the quadrics containing $C_P$ is $C_P$ union all its trisecant lines as soon as $\Cliff(C_P) \geq 4$. 
Setting $k_0 = \gon(C_P)$, by \cite[Thm. 1.1]{kl3}, we therefore need to show, in both cases, that $k_0 \geq 6$.

Suppose now that $\phi(L) \geq 4$. By \cite[Cor. 1.6 and Cor. 1.5]{kl1}, we are done except possibly when $L^2 \leq 22$ and $\phi(L) = 4$. If $L^2 = 20$ or $22$, by \cite[Cor. 1.6 and Cor. 1.5]{kl1}, we have $k = 7$, $k_0 \geq 5$ and if equality holds, by Proposition \ref{more}(i)-(ii), there exists a divisor $D$ on $S$ with $D^2 = 4, D.L=9$. The case $L^2 = 22$ contradicts the Hodge index theorem, while if $L^2 = 20$ we get $(L - 2D)^2 = 0$ and the contradiction $4 = \phi(L) \leq L.(L-2D) = 2$. If $L^2 = 18$, by \cite[Cor. 1.6 and Cor. 1.5]{kl1}, we have $k = 6$, $k_0 \geq 4$ and if equality holds, by Proposition \ref{more}(i)-(ii), there exists a divisor $D$ on $S$ with $D^2 = 4, D.L = 8$, contradicting the Hodge index theorem. Hence, to prove (ii), we are left with the cases $(L^2, \phi(L), k_0) = (18, 4, 5), (16, 4, 5)$ and $(16, 4, 4)$.

Now suppose that $\phi(L) = 3$ so that $k = 6$, $k_0 \geq 4$ by  \cite[Cor. 1.6 and Cor. 1.5]{kl1} and assume that $k_0 = 4, 5$ with $(k_0, L^2) \neq (5, 18)$. By Proposition \ref{more} there exists a divisor $D$ on $S$ with either $D^2 = 2, D.L = 6$ or $D^2 = 4, D.L = 8$ when $k_0 = 4$ and $D^2 = 2, D.L\leq 7$ or $D^2 = 4, D.L\leq 9$ when $k_0 = 5$.  Then the Hodge index implies that either $k_0 = 4$ and $L^2= 18$ or $k_0 = 5$ and $20 \leq L^2 \leq 24$. By \cite[Prop. 3.1]{kl1}(a) we have a decomposition $L \sim N + N'$ with $N'$ base-component free, $2 (N')^2 \leq L.N' \leq (N')^2  + k_0 \leq 2k_0$ and if $\phi(N') = 1$ and $L.N' \geq (N')^2 + k_0 - 1$ then $\Bs |N'| \cap C_P \not= \emptyset$. If $(N')^2 = 0$ then $N' \sim mB$ for some $m \geq 1$ and a genus one pencil $B$, giving the contradiction $k_0 \geq L.N' \geq 2 \phi(L) = 6$. If $(N')^2 = 2$  then $L.N' \geq 2 \phi(L) = 6 \geq k_0 + 1 = (N')^2 + k_0 - 1$, whence $\Bs |N'| \cap C_P \not= \emptyset$, a contradiction because $C_P$ is cut out by a general hyperplane passing through $P \not\in S$. Therefore $(N')^2 = 4$ and the Hodge index theorem applied to $N'$ and $L$ shows that we are in the case $k_0 = 5$, $L^2= 20$ and $L.N' = 9$. But this gives $N.N'=5$, $N^2 = 6$, $L.N = 11$, whence $(N-N')^2 = 0$ and the contradiction $3 = \phi(L) \leq L.(N-N') = 2$. Hence, to prove (iii), we are left with the case $(L^2, \phi(L), k_0) = (18, 3, 5)$.

Note that if $(L^2, \phi(L)) = (18, 4)$, by \cite[Prop. 1.4 and Lemma 2.14]{kl1}, there exist, $E>0, E_i > 0$ for $i = 1, 2$, such that $E^2 = E_i^2 = 0$, $E.E_1 = E.E_2 = 2$, $E_1.E_2 = 1$ and $L \sim 2 E + E_1 + E_2$. Moreover we claim that there exists a divisor $F>0$ be such that $F^2 = 0, F.L=5$ and $P \not \in \Supp(F)$. To see the latter, since $E_1.L = E_2.L = 5$, we can choose as $F$ one among $E_1$, $E_1+K_S$, $E_2$, $E_2+K_S$ and we just need to show that their intersection is empty. This is certainly true and well-known if either $E_1$ or $E_2$ are nef. If not, by \cite[Lemma 2.3]{kl1}, for $i = 1, 2$, there are nodal curves $R_i$ such that, setting $l_i = - R_i.E_i \geq 1$, there exists an $A_i > 0$ with $A_i^2 = 0$ and $E_i \sim A_i + l_i R_i$. Now $5 = E_i.L =  A_i.L + l_i R_i.L \geq 5$ shows that $A_i.L = 4$, $l_i = 1$ and $A_i.R_i= 1$. Again by \cite[Lemma 2.3]{kl1}, $A_i$ is nef for $i = 1, 2$. Then $\Supp(A_i) \cap \Supp(A_i + K_S) = \emptyset$ and it remains to show that $R_1.R_2=0$. It cannot be $E.A_i = 0$, for then, by \cite[Lemma 2.1]{klvan}, $E \eqv A_i$ giving the contradiction $2 = E.E_1 = E.R_i = A_i.R_i= 1$. From $4 = L.A_i = 2E.A_i  + E_1.A_i +  E_2.A_i \geq  4$ we deduce that $E.A_i=E_j.A_i=1$ for all $i, j$ and $R_1.R_2 = (E_1-A_1).(E_2-A_2) = -1 + A_1.A_2$. Therefore it cannot be $A_1.A_2=0$, for then $R_1.R_2 = -1$, whence $1 = A_1.E_2 = A_1.A_2 + A_1.R_2 \geq 1$ gives $A_1.A_2=1$ and then  $R_1.R_2=0$.

We now treat together the two cases $(L^2, \phi(L), k_0) = (18, 4, 5)$ or $L$ is as in Lemma \ref{18} (i) and $k_0=5$. Let $A$ be a $g^1_5$ on $C_P$ and let us apply \cite[Prop. 3.1]{kl1}.  

In case (a) of that Proposition, we have a decomposition $L \sim N + N'$ with $N'$ base-component free, $2 (N')^2 \leq L.N' \leq (N')^2  + 5 \leq 10$, $N'_{|C_P} \geq A$, either $N \geq N'$ or $N$ is base-component free and $\Bs |N| \subset C_P$. Moreover if $\phi(N') = 1$ and $L.N' \geq (N')^2 + 4$, then $\Bs |N'| \cap C_P \neq \emptyset$. Now it cannot be $(N')^2 = 0$ for then, as usual, we get the contradiction $5 \geq L.N' \geq 2 \phi(L) \geq 6$. Also we cannot have $(N')^2 = 2$ for then if $\phi(L) =4$ we get the contradiction $7 \geq L.N' \geq 2 \phi(L) = 8$, while if $\phi(L) = 3$, since $L.N' \geq 2 \phi(L) = 6 \geq (N')^2 + 4$, we get the contradiction $\Bs |N'| \cap C_P \neq \emptyset$. Therefore $(N')^2 = 4$ and the Hodge index theorem gives $N'.L=9$, $N.N' = 5$ and $N^2 = 4$. Now $L.(N-N')=0$ and $(N-N')^2 = -2$ show that $H^0(N-N')=H^1(N-N')=0$. Then $N$ is base-component free, whence $H^1(N)=H^1(N+K_S)=0$.

Let $Z \in |A|$ and let $Z' \in |N'_{|C_P} - Z|$. The exact sequence
\begin{equation}
\label{si} 
0 \hpil -N \hpil N' \hpil (N')_{|C_P} \hpil 0 
\end{equation}
shows that there exists $D' \in |N'|$ such that $D'_{|C_P} = Z + Z'$. Now the exact sequence
\begin{equation}
\label{eq} 
0 \hpil -N \hpil \I_{Z'/S}(N') \hpil A \hpil 0 
\end{equation}
gives $h^0(\I_{Z'/S}(N')) = 2$. Let 
\begin{equation}
\label{v} 
|H^0(\I_{Z'/S}(N'))| = |V| + G
\end{equation}
be the decomposition into moving and fixed part, so that $|V|$ is base-component free and $G \geq 0$. Let $M = N' - G$ and pick a general divisor $D_0 \in |H^0(\I_{Z'/S}(N'))|$, so that there exist $M', M_0 \in |V|$ such that $D' = M' + G$,  $D_0 = M_0 + G$ and $M_0$ is general in $|V|$. In particular $M'$ and $M_0$ have no common components. Now 
\begin{equation}
\label{z'} 
Z' \sub D' \cap D_0 = (M' \cap M_0) \cup G \ \mbox{and} \ \length(Z') = N'.L - \deg Z = (N')^2.
\end{equation}
If $G > 0$, as $Z' \subset C_P$, \eqref{z'} gives $L.G \geq 4 - M^2$. From $9 = L.N' = L.M + L.G \geq L.M + 1$ we get $L.M \leq 8$ and the Hodge index theorem implies that $M^2 \leq 2$. Therefore $L.G \geq 2$ and, as $M$ is base-component free and nontrivial, we have $L.M \geq 2\phi(L)$. When $\phi(L) = 4$ we get the contradiction $9 = L.M + L.G \geq 10$. When $\phi(L) = 3$ we have $L.M \geq 6$, whence, as $L$ is 3-divisible, $L.G = 3, M^2=2$ and $L.M=6$. The Hodge index theorem gives $L \eqv 3M$, whence $M.G=1$. Now $4 = (N')^2 = M^2 + 2M.G + G^2$ gives $G^2=0$ and $G.N' = 1$, so that we get the contradiction $\Bs |N'| \cap C_P \neq \emptyset$. 

Hence $G = 0$ and $\length(Z') = (N')^2 = 4 = M^2$ and therefore $Z' = D' \cap D_0$, whence $Z'$ is a Cartier divisor on $D'$. Now the exact sequence
\begin{equation}
\label{n'} 
0 \hpil \O_S \hpil \I_{Z'/S}(N') \hpil \I_{Z'/D'}(N') \hpil 0 
\end{equation}
shows that $h^0(\I_{Z'/D'}(N')) = 1$. But  $\deg N'_{|D'}(-Z') = 0$, whence, on $D'$ we have $N'_{|D'} \sim Z'$ and therefore $Z \sim L_{|D'} - Z' \sim N_{|D'} + (N')_{|D'} - Z' \sim N_{|D'}$.

We have therefore proved that there exists a decomposition $L \sim N + N'$ with $N'$ base-component free, $N.N'=5$, $H^1(N)=$ $H^1(N+K_S)=H^0(N-N')=H^1(N-N')=0$, $(N')^2 = N^2 = 4$, a divisor $D' \in |N'|$ and a divisor $Z \in |N_{|D'}|$ such that $Z \subset C_P$. Let $X$ be the family of such pairs $(D', Z)$ and consider the incidence subvariety of $|L| \times X$:
\begin{equation*}
\label{jm}
\J = \{ (C,(D', Z)) \; : \; C \in |L|, (D', Z) \in X, Z \subset C \}
\end{equation*}
together with its two projections 
\[  
\xymatrix{& \J \ \ \ar[dl]_{\pi_1}
\ar[dr]^{\pi_2} \ar@{^{(}->}[r] & |L| \times X \\
|L| & & X.}
\]
Of course we have proved that $C_P \in \Im \pi_1$. One easily checks that $\dim X = \dim |N'| + \dim |N_{|D'}| = 4$ and the dimension of a general fiber of $\pi_2$ is at most $h^0(\I_{Z/S}(L)) - 1$. The exact sequence
\begin{equation} 
\label{enne}
0 \hpil N \hpil \I_{Z/S}(L) \hpil \I_{Z/D'}(L) \hpil 0 
\end{equation}
now shows that $h^0(\I_{Z/S}(L)) = h^0(N) + h^0((N')_{|D'}) = 5$, so that $\dim \J \leq 8$. On the other hand let $C \in \Im \pi_1$ be general, so that there exist $D' \in |N'|$ and $Z \in |N_{|D'}|$ such that $Z \subset C$. Therefore $Z \subset D' \cap C$ and we can find an effective divisor $Z'$ such that $Z' \sim D'_{|C} - Z$ on $C$. Note that $Z$ is Cartier on $D'$, whence, on $D'$, we find $Z' \sim L_{|D'} - Z \sim N'_{|D'}$. Now \eqref{n'} gives $h^0(\I_{Z'/S}(N')) = 1 + h^0(\O_{D'}) \geq 2$ and from the exact sequence
\begin{equation*} 
0 \hpil -N \hpil \I_{Z'/S}(N') \hpil  \I_{Z'/C}(N') \hpil 0 
\end{equation*}
we get $H^0(\I_{Z'/S}(N')) \cong  H^0(\I_{Z'/C}(N'))$ and $h^0(\I_{Z'/C}(N')) \geq 2$. Pick a general divisor $Z'' \in |\I_{Z'/C}(N')|$, so that there is a divisor $D'' \in |N'|$ such that $Z' \subset D''$, $D'' \cap C = Z'' + Z'$ and $D''$ is general in $|H^0(\I_{Z'/S}(N'))|$. Exactly as we did above we deduce that $Z' = D' \cap D''$, whence $Z'$ is a Cartier divisor on $D''$. From the exact sequence
\begin{equation*} 
0 \hpil \O_S \hpil \I_{Z'/S}(N') \hpil  \I_{Z'/D''}(N') \hpil 0 
\end{equation*}
we get $h^0( \I_{Z'/D''}(N')) \geq 1$. Since $\deg N'_{|D''} = \deg Z'$ we find $Z' \sim N'_{|D''}$ on $D''$ and therefore $Z'' \sim N_{|D''}$. This shows that the at least one-dimensional family of pairs $(Z'', D'')$ belong to $X$ and therefore any fiber of $\pi_1$ has dimension at least one, so that $\dim \Im \pi_1 \leq 7$. As $\dim |L| = 9$ and the possible decompositions $L \sim N + N'$ as above are finitely many, we have shown that this case cannot occur.

Therefore we must be in case (b) of \cite[Prop. 3.1]{kl1}. Let $F>0$ be such that $F^2 = 0, F.L=5$ and $P \not \in \Supp(F)$ in case $\phi(L) = 4$ or $F = F_1$ or $E$ in case $\phi(L) = 3$. Since $h^1(A) = 6 > F.L$, we find as usual that $h^0(\E(C_P,A)(-F))>0$, whence we have a decomposition $L \sim N + N'$ with $N \geq F$, $N' = N + R + K_S$ base-component free, $R$ a nodal cycle, $N.N' = 5$, $N'_{|C_P} \geq A$ and $\Bs |N'| \subset C_P$. It cannot be $(N')^2 = 0$ for then $5 = N.N' +  (N')^2 = L.N' \geq 2 \phi(L) \geq 6$. Also we cannot have $\phi(N') = 1$ (and, in particular, $(N')^2 = 2$) for then we have the contradiction $\emptyset \not= \Bs |N'| \subset C_P$. Hence $N'$ is base-point free and $(N')^2 \geq 4$. From $N^2 + 5 = L.N \geq L.F$ we see that $N^2 \geq 0$ when $\phi(L) = 4$ and $N^2 \geq -2$ when $\phi(L) = 3$. But if $N^2 = -2$ we get $L.N = 3 \geq L.F \geq 3$ so that $N=F$ giving the contradiction $F^2 = -2$. Therefore $N^2 \geq 0$ and from $18 = L^2 = N^2 + (N')^2 + 10$ we deduce the three possibilities $(N^2, (N')^2) = (4, 4), (2, 6)$ or $(0, 8)$. In the first case we get the contradiction $L.R = L.N'-L.N = 0$. If $(N^2, (N')^2) = (2, 6)$ we have $L.N = 7$ which is not divisible by 3, whence $\phi(L) = 4$, giving the contradiction $7 = L.N \geq 2\phi(L) = 8$. If $(N^2, (N')^2) = (0, 8)$  we have $L.N = 5$, whence again $\phi(L) = 4$ and $5 = L.N \geq L.F = 5$, so that $N = F$. 

Hence the proof in the case $\phi(L) = 3$ and $L$ as in Lemma \ref{18}(i) is completed. 

When $\phi(L) = 4$ we note that $H^1(N)=H^1(N+K_S)=0$ since $L.N = 5 < 2\phi(L)$ (see for example \cite[Lemma 2.5]{klm}). We also record that, since $N-N'=-R+K_S$, we have $H^0(N-N')=0$ and $h^1(N-N')=1$.

Let $Z \in |A|$ and let $Z' \in |N'_{|C_P} - Z|$. As above there exists $D' \in |N'|$ such that $D'_{|C_P} = Z + Z'$ and again \eqref{eq} gives $h^0(\I_{Z'/S}(N')) = 2$. With the same notation as in \eqref{v} we now deduce by \eqref{z'} that, if $G > 0$, then $L.G \geq 8 - M^2$. From $13 = L.N' = L.M + L.G \geq L.M + 1$ we get $L.M \leq 12$ and the Hodge index theorem implies that $M^2 \leq 8$. If equality holds we get $2L \eqv 3M$ and the contradiction $10 = 2L.E_1 = 3M.E_1$. If  $M^2= 6$ then $L.G \geq 2$ and $L.M \geq 3 \phi(L) = 12$, giving the contradiction $L.M + L.G \geq 14$. If  $M^2 = 2, 0$ then $L.G \geq 6, 8$ and $L.M \geq 2 \phi(L) = 8$, giving the same contradiction. If $M^2= 4$ then $L.G \geq 4, L.M = 13 - L.G \leq 9$ and the Hodge index theorem gives $L.M = 9$ and $L.G = 4$. We also prove, for later, that $H^1(L-M)=H^0(L-2M)=H^1(L-2M)=0$. Suppose first that $h^1(L-M) > 0$. Then, by Riemann-Roch, as $(L-M)^2 = 4$, we get $h^0(L-M) \geq 4$ and we can write $|L-M| = |B| + \Gamma$ with $|B|$ base-component free, $\Gamma \geq 0$ the fixed component and  $h^0(B) \geq 4$. Also $L.B \leq L.(L-M) = 9$. But then either $B^2 = 0$ and $B \sim hQ$ for a genus one pencil $Q$ and $h \geq 3$, giving the contradiction $9 \geq L.B \geq 6 \phi(L) = 24$ or $B^2 > 0$ so that $H^1(B) = 0$ and Riemann-Roch gives $B^2 \geq 6$, contradicting the Hodge index theorem. This proves that  $H^1(L-M)=0$. Now $(L-2M)^2 = -2$ and $L.(L-2M) = 0$ easily gives $H^0(L-2M)=H^1(L-2M)=0$.

Suppose first that $G = 0$.

By \eqref{z'} we have $Z' = D' \cap D_0$, whence $Z'$ is a Cartier divisor on $D'$. Now \eqref{n'} gives
$h^0(\I_{Z'/D'}(N')) = 1$. But  $\deg N'_{|D'}(-Z') = 0$, whence, on $D'$ we have $N'_{|D'} \sim Z'$ and therefore $Z \sim L_{|D'} - Z' \sim N_{|D'} + (N')_{|D'} - Z' \sim N_{|D'}$.

We have therefore proved that there exists a decomposition $L \sim N + N'$ with $N'$ base-point free, $N.N'=5$,  $(N^2, (N')^2) = (0, 8)$, $H^1(N)=H^1(N+K_S)=0$, $H^0(N-N')=0$, $h^1(N-N')=1$, a divisor $D' \in |N'|$ and a divisor $Z \in |N_{|D'}|$ such that $Z \subset C_P$. Let $X$ be the family of such pairs $(D', Z)$ and consider the incidence subvariety of $|V_P| \times X$:
\begin{equation}
\label{jm1}
\J = \{ (C,(D', Z)) \; : \; C \in |V_P|, (D', Z) \in X, Z \subset C \}
\end{equation}
together with its two projections 
\[  
\xymatrix{& \J \ \ \ar[dl]_{\pi_1}
\ar[dr]^{\pi_2} \ar@{^{(}->}[r] & |V_P| \times X \\ |V_P| & & X.}
\]
Of course, as $C_P \in \Im \pi_1$, we have proved that $\pi_1$ is surjective. One easily checks that $X$ is irreducible and $\dim X = \dim |N'| + \dim |N_{|D'}| = 5$. As $Z$ moves on $C_P$, we see that a general element $(D', Z) \in X$ is such that $P \not\in Z$, and the dimension of a general fiber of $\pi_2$ is at most $\dim \{T \in |V_P| : Z \subset T \} = h^0(\I_{Z \cup \{P\}/S}(L)) - 1$. Since $P \not \in \Supp(F)$, then either $P \not \in \Supp(D')$ and we find an effective divisor $D' + F \sim L$ such that $P \not \in \Supp(D' + F)$ and $Z \subset \Supp(D' + F)$ or $P \in \Supp(D')$ and the natural map $H^0(\I_{Z/D'}(L)) \to H^0(L_{|\{P\}})$ is surjective since $\I_{Z/D'}(L)\cong N'_{|D'}$ is base-point free and from \eqref{enne} we get
that $H^0(\I_{Z/S}(L))\to H^0(\I_{Z/D'}(L))$ is surjective. Therefore $h^0(\I_{Z \cup \{P\}/S}(L)) - 1= h^0(\I_{Z/S}(L)) - 2$ and \eqref{enne}
gives $h^0(\I_{Z/S}(L)) =  h^0(N) + h^0(N'_{|D'}) =  5$, so that the dimension of a general fiber of $\pi_2$ is at most 3 and $\dim \J \leq 8$. On the other hand, exactly as in case (a) above, any fiber of $\pi_1$ has dimension at least one, so that $\dim \Im \pi_1 \leq 7$, a contradiction.

Suppose now that $G > 0$, so that, as we have seen above, $M^2= 4, L.M = 9$ and $L.G = 4$. Recall that $Z' \subset D' \cap C_P$ with $D' = M' + G$. Let $Z_G = G \cap Z' \sub G \cap C_P$ be the scheme-theoretic intersection and let $Z_M = Z' - Z_G$, seen as effective divisors on $C_P$. Then $Z_M \cap G = \emptyset$ and \eqref{z'} shows that $Z_M \sub M' \cap M_0$. On the other hand $M^2 \geq \length(Z_M) =  \length(Z') -  \length(Z_G)  \geq \length(Z') - L.G = N'.L - 9 = 4 = M^2$ and therefore $Z_M = M' \cap M_0$, $Z_G = G \cap C_P$ and, in particular, $Z_M$ is a Cartier divisor on $M'$. Moreover on $C_P$ we have $Z + Z_M + Z_G = Z + Z' = D' \cap C_P = M' \cap C_P + G \cap C_P$, so that $Z + Z_M = M' \cap C_P$. Also note that, on $M'$, we have $Z \sim L_{|M'} - Z_M \sim (L-M)_{|M'}$.

We have therefore proved that there exists a decomposition $L \sim M + G + N$ with $N>0$, $M$ base-component free,  $M^2=4, N^2 = 0$, $H^1(L-M)=H^0(L-2M)=H^1(L-2M)=0$, a divisor $M' \in |M|$ and a divisor $Z \in |(L-M)_{|M'}|$ such that $Z \subset C_P$. Let $X$ be the family of such pairs $(M', Z)$ and consider the incidence subvariety of $|V_P| \times X$:
\begin{equation*}
\J = \{ (C,(M', Z)) \; : \; C \in |V_P|, (M', Z) \in X, Z \subset C \}
\end{equation*}
together with its two projections 
\[  
\xymatrix{& \J \ \ \ar[dl]_{\pi_1}
\ar[dr]^{\pi_2} \ar@{^{(}->}[r] & |V_P| \times X \\ |V_P| & & X.}
\]
Of course, as $C_P \in \Im \pi_1$, we have proved that $\pi_1$ is surjective. One easily checks that $X$ is irreducible and $\dim X = \dim |M| + \dim |(L-M)_{|M'}| = 4$.  Now, exactly as before, the dimension of a general fiber of $\pi_2$ is at most $h^0(\I_{Z \cup \{P\}/S}(L)) - 1 = h^0(\I_{Z/S}(L)) - 2 =  3$, so that $\dim \J \leq 7$ and $\dim \Im \pi_1 \leq 7$, a contradiction.

We now deal with the case $(L^2, \phi(L), k_0) = (18, 3, 5)$ and $L$ is as in Lemma \ref{18} (ii) and not of special type (see Def. \ref{18sp}). Let $A$ be a $g^1_5$ on $C_P$ and let $\E:=\E(C,A)$ be the vector bundle as defined in \cite{kl1}. Consider the decomposition $L+K_S \sim D_1+D_2$ with $D_1= E+E_1$ and $D_2 = E+E_2+E_3+K_S$. Since $D_1^2+D_2^2 =8$, we get by \cite[Lemma 3.2]{kl1}, that $h^0(\E(-D_0))>0$ with $D_0=D_1$ or $D_0=D_2$ and $(L-D_0)_{|C_P} \geq A$. Saturating we get an exact sequence
\begin{equation}
\label{ecco}
0 \hpil N \hpil \E \hpil N' \* \I_X \hpil 0
\end{equation}
with $N, N' \in \Pic S$, $X$ a zero-dimensional scheme, $L \sim N + N'$ and $N \geq D_0$. Moreover $N'$ is base component free and nontrivial, $\Bs|N'| \subset C_P \cup X$ and $(N')_{|C_P} \geq A$. 

Since $N' \leq L-D_0$ and $h^1(L-D_0)=0$ we get $(N')^2 \leq (L-D_0)^2 \leq 6$. Also, taking $c_2$ in \eqref{ecco}, we find that $N.N' \leq N.N' + \length(X) =5$.

Now if $(N')^2 =0$ we get the contradiction $6 = 2 \phi(L) \leq L.N' = N.N' \leq 5$. If $(N')^2 =2$ then $|N'|$ has two base points, whence, as $C_P$ is general, $\Bs|N'| \cap C_P = \emptyset$ and therefore $\length(X) \geq 2$ and $N.N' \leq 3$, giving the  contradiction $6 = 2 \phi(L) \leq L.N' = 2 + N.N' \leq 5$.
If $(N')^2 =4$ we have $L.N' = 4 + N.N' \leq 9$ and then $L.N'=9$ by the Hodge index theorem. Therefore $N.N'=5$ and $X = \emptyset$. As above it cannot be $\phi(N') = 1$ for then $|N'|$ has two base points which must lie on $C_P$ and therefore $\phi(N') \geq 2$ giving the contradiction $9 = L.N' \geq 5 \phi(N') \geq 10$.
Hence $(N')^2 =6$ and then $D_0 =  E+E_2+E_3+K_S$, so that $N' \leq E+E_2+E_3$. But $11 = L.(E+E_2+E_3) \geq L.N' \geq 11$ by the Hodeg index theorem and therefore $N' = E+E_2+E_3$ and $N = E + E_1$. Let $Z \in |A|$ and let $Z' \in |N'_{|C_P} - Z|$. As above, \eqref{si} shows that there exists $D' \in |N'|$ such that $D'_{|C_P} = Z + Z'$ and \eqref{eq} gives $h^0(\I_{Z'/S}(N')) = 2$. With notation as in \eqref{v} we get $N' = M + G$ with $|M|$ base-component free nontrivial and $G \geq 0$. We now claim that $G = 0$. If $G > 0$ we get by \eqref{z'} that $L.G \geq 6 - M^2$. From $11 = L.N' = L.M + L.G$ we get $L.M \leq 10$ and the Hodge index theorem implies that $M^2 \leq 4$. If $M^2=0$ we have $L.G \geq 6$ and the contradiction $5 \geq L.M \geq 2\phi(L) = 6$. If $M^2 = 2$ we can write $M = F_1 + F_2$ with $F_i > 0, F_i^2=0$ and $F_1.F_2=1$. Now $L.G \geq 4$, whence $7 \geq L.M = L.F_1 + L.F_2.$ Since $E$ is (numerically) the only $F>0$ such that $F^2=0, L.F=3$ and the $E_j$'s are (numerically) the only $F>0$ such that $F^2=0, L.F=4$, we deduce that $F_1 \eqv E, F_2 \eqv E_j$ for some $j$ such that $1 \leq j \leq 3$. Then $M \eqv E+E_j$ and $G \eqv E_2 + E_3-E_j$. It cannot be $j = 1$, for this case is excluded by the hypothesis that $L$ not of special type, whence we can assume that $j=2$, $M \eqv E+E_2$ and $G \eqv E_3$.  Recall that $Z' \subset D' \cap C_P$ with $D' = M' + G$. Let $Z_G = G \cap Z' \sub G \cap C_P$ be the scheme-theoretic intersection and let $Z_M = Z' - Z_G$, seen as effective divisors on $C_P$. Then $Z_M \cap G = \emptyset$ and \eqref{z'} shows that $Z_M \sub M' \cap M_0$. On the other hand $M^2 \geq \length(Z_M) =  \length(Z') -  \length(Z_G)  \geq \length(Z') - L.G = 2 = M^2$ and therefore $Z_M = M' \cap M_0$, $Z_G = G \cap C_P$ and, in particular, $Z_M$ is a Cartier divisor on $M'$. Moreover on $C_P$ we have $Z + Z_M + Z_G = Z + Z' = D' \cap C_P = M' \cap C_P + G \cap C_P$, so that $Z + Z_M = M' \cap C_P$. But now $M_{|C_P} \sim A + Z_M$ and $h^0(M_{|C_P}) = h^0(M) = 2$ and then $Z_M = \Bs|M| \subset C_P$, a contradiction.
Therefore $M^2 = 4, L.G \geq 2$ and $L.M \leq 9$. Now either $\phi(M) = 1$, but then we can write $M = 2F_1 + F_2$ with $F_i > 0, F_i^2=0, F_1.F_2=1$ and we get the contradiction $9 \geq L.M = 2L.F_1 + L.F_2 \geq 10$ or  $\phi(M) \geq 2$, but then $M.N \geq 2\phi(M) \geq 4$, $M.N' \geq 3\phi(M) \geq 6$ so that $M.G \geq 2$, giving the contradiction $9 \geq L.M = M^2 + M.G + M.N \geq 10$. This gives that $G = 0$.

We have therefore proved that there exists a decomposition $L \sim N + N'$ with $N'$ base-component free, $(N')^2=6, N^2=2$, $H^1(N)=0$ and $H^1(N-N') = H^1(E_2 + E_3+K_S - E_1) = 0$, a divisor $D' \in |N'|$ and a divisor $Z \in |N_{|D'}|$ such that $Z \subset C_P$. Let $X$ be the family of such pairs $(D', Z)$ and consider the incidence subvariety in \eqref{jm1} together with its two projections. Of course we have proved that $C_P \in \Im \pi_1$, that is $\pi_1$ is surjective. One easily checks that $\dim X = \dim |N'| + \dim |N_{|D'}| =4$ and, using \eqref{enne}, the dimension of a general fiber of $\pi_2$ is at most $4$. On the other hand, exactly as above, any fiber of $\pi_1$ has dimension at least one, so that $\dim \Im \pi_1 \leq 7$, a contradiction.

This concludes the proof for the case $L^2 = 18$.

We now treat the cases $L^2=16, \phi(L)=4$ and $k_0 = 4, 5$. Let $A$ be a $g^1_{k_0}$ on $C_P$.

If $k_0 = 4$ we are in case (a) of \cite[Prop. 3.1]{kl1}, so that there is a decomposition $L \sim N + N'$ with $N'$ base-component free, $2 (N')^2 \leq L.N' \leq (N')^2  + 4 \leq 8$ and $N'_{|C_P} \geq A$. Moreover if $\phi(N') = 1$ and $L.N' \geq (N')^2 + 3$, then $\Bs |N'| \cap C_P \neq \emptyset$. Now it cannot be $(N')^2 = 0$ for this gives the contradiction $4 \geq L.N' \geq 2 \phi(L) \geq 8$. Also we cannot have $(N')^2 = 2$ for then $6 \geq L.N' \geq 2 \phi(L) \geq 8$. Therefore $(N')^2 = 4$ and the Hodge index theorem gives $L \eqv 2N', N \eqv N'$ and $H^1(N)=H^1(N+K_S)=0$.

Let $Z \in |A|$ and let $Z' \in |N'_{|C_P} - Z|$. As above, \eqref{si} shows that there exists $D' \in |N'|$ such that $D'_{|C_P} = Z + Z'$ and \eqref{eq}
gives $h^0(\I_{Z'/S}(N')) = 2$. With notation as in \eqref{v} we get $N' = M + G$ with $|M|$ base-component free nontrivial and $G \geq 0$. From $8 = L.N' = L.M + L.G \geq 2\phi(L)+ L.G \geq 8$ we get that $G = 0$ and \eqref{z'} gives $\length(Z') = (N')^2 = 4 = M^2$. Therefore $Z' = D' \cap D_0$, whence $Z'$ is a Cartier divisor on $D'$. Now \eqref{n'} gives $h^0(\I_{Z'/D'}(N')) = 1$. But  $\deg N'_{|D'}(-Z') = 0$, whence, on $D'$ we have $N'_{|D'} \sim Z'$ and therefore $Z \sim N_{|D'}$.

We have therefore proved that there exists a decomposition $L \sim N + N'$ with $N'$ base-component free, $N \eqv N'$, $H^1(N)=$ $H^1(N+K_S)=0$, $(N')^2 = 4$, a divisor $D' \in |N'|$ and a divisor $Z \in |N_{|D'}|$ such that $Z \subset C_P$. Let $X$ be the family of such pairs $(D', Z)$ and consider the incidence subvariety in \eqref{jm1} together with its two projections. Of course we have proved that $C_P \in \Im \pi_1$, that is $\pi_1$ is surjective. One easily checks that $\dim X = \dim |N'| + \dim |N_{|D'}| \leq 4$ and the dimension of a general fiber of $\pi_2$ is at most $h^0(\I_{Z \cup \{P\}/S}(L)) - 1$. Now either $P \not \in \Supp(D')$ and, as $N$ is base-point free, we find an effective divisor $D' + N_1 \sim L$, $N_1 \in |N|$, such that $P \not \in \Supp(D' + N_1)$ and $Z \subset \Supp(D' + N_1)$ or $P \in \Supp(D')$ and the natural map $H^0(\I_{Z/D'}(L)) \to H^0(L_{|\{P\}})$ is surjective since $\I_{Z/D'}(L)\cong N'_{|D'}$ is base-point free and from \eqref{enne} we get that $H^0(\I_{Z/S}(L))\to H^0(\I_{Z/D'}(L))$ is surjective. Therefore $h^0(\I_{Z \cup \{P\}/S}(L)) = h^0(\I_{Z/S}(L)) - 1$ and \eqref{enne} gives $h^0(\I_{Z/S}(L)) = h^0(N) + h^0(N'_{|D'}) =  5$, so that the dimension of a general fiber of $\pi_2$ is at most 3 and $\dim \J \leq 7$.  On the other hand, exactly as above, any fiber of $\pi_1$ has dimension at least one, so that $\dim \Im \pi_1 \leq 6$, a contradiction.

This proves that the case $L^2=16, \phi(L)= 4$ and $k_0 = 4$ does not occur. 

Assume now $L^2=16, \phi(L)=4$ and $k_0 = 5$. Recall that, by \cite[Prop. 1.4]{kl1}, there exist $E>0, E_1> 0$ such that $E^2 = E_1^2 = 0$, $E.E_1 = 2$ and $L \eqv 2 (E + E_1)$.  Set $D_1 = E + E_1$ and $D_2 = L+K_S - D_1$, so that $h^1(D_1+K_S)=0$ and \cite[Lemma 3.2]{kl1} gives a divisor $D_0  \eqv E + E_1$ such that $h^0(\E(-D_0))>0$, where $\E=\E(C_P,A)$.

As in the proof of \cite[Lemma 2.4]{kl3} we have an exact sequence
\begin{equation*}
0 \hpil N \hpil \E \hpil N' \* \I_W \hpil 0  
\end{equation*}
with $N, N' \in \Pic S$, $W$ a zero-dimensional scheme, $N \geq D_0$ and $L \sim N + N'$. Moreover $N'$ is base-component free and nontrivial, $N'_{|C_P} \geq A$ and  $N.N' \leq N.N' + \length(W) =5$.

Since $N' \leq D_0$ and both are nef, we get $(N')^2 \leq D_0^2 = 4$. If $(N')^2 =0, 2$ we get the contradiction $8 =  2 \phi(L) \leq L.N' = (N')^2 + N.N' \leq 7$. Therefore $(N')^2 = 4$ and $L.N' = (N')^2 + N.N' \leq 9$, whence, as $L$ is 2-divisible, $L.N'  \leq 8$. Now the Hodge index theorem gives $L \eqv 2N'$ and this case can be excluded exactly as the previous one.

This concludes the proof for the case $L^2 = 16, \phi(L)=4$.

\end{proof}
\renewcommand{\proofname}{Proof}

\begin{rem}
\label{rm}

\rm{With our methods, we could also extend Theorem \ref{ideale}(iii) when $(L^2, \phi(L)) = (16, 3)$ or $(18, 3)$. Here, for the general hyperplane section $C_{\eta}$ of $S$, we have that $\gon(C_{\eta}) = 6$ and, as above, if $C_P$ is cut out on $S$ by a general hyperplane section of passing through a point $P \in \PP^r - S$, one easily excludes the case $\gon(C_P) = 4$, but we do not know if it can be $\gon(C_P) = 5$. If this can be excluded then (iii) holds for $S$ also in these cases.}
\end{rem}

\end{document}